\documentclass[12pt,a4paper]{article}
\usepackage{stmaryrd}
\usepackage{booktabs}
\usepackage{subfigure}
\usepackage{epsf,epsfig,amsfonts,amsgen,amsmath,amstext,amsbsy,amsopn,amsthm}
\usepackage{amsmath,times,mathptmx}

\usepackage{amsfonts,amsthm,amssymb,mathrsfs,bbding}
\usepackage{txfonts}
\usepackage{latexsym,bm}
\usepackage{indentfirst}
\usepackage{graphicx}
\usepackage{graphics,multicol}
\usepackage{color}
\usepackage[colorlinks=true,anchorcolor=blue,filecolor=blue,linkcolor=blue,urlcolor=blue,citecolor=blue]{hyperref}
\usepackage{floatrow}
\usepackage{enumerate}
\usepackage{tikz}
\usetikzlibrary{calc}
\evensidemargin=\oddsidemargin\topmargin=-1.5cm

\newtheorem{thm}{Theorem}[section]
\newtheorem{prob}{Problem}[section]

\newtheorem{lem}{Lemma}[section]
\newtheorem{cor}{Corollary}[section]

\newtheorem{conj}{Conjecture}[section]

\newtheorem{claim}{Claim}[section]

\newtheorem{case}{Case}
\newtheorem{subcase}{Subcase}[case]

\makeatletter
\@addtoreset{case}{lem}
\makeatother

\makeatletter
\@addtoreset{claim}{section}
\makeatother

\allowdisplaybreaks\allowdisplaybreaks[4]

\addtocounter{section}{0}

\begin{document}
\title{Long cycles and spectral radii in planar graphs\footnote{Supported by the National Natural Science Foundation of China (Nos.\,12271162,\,12326372), and Natural Science Foundation of Shanghai (Nos. 22ZR1416300 and 23JC1401500) and The Program for Professor of Special Appointment (Eastern Scholar) at Shanghai Institutions of Higher Learning (No. TP2022031).}}

\author{ {\bf Ping Xu$^{a,b}$},
{\bf Huiqiu Lin$^{a}$}\thanks{Corresponding author: huiqiulin@126.com(H. Lin)}, {\bf Longfei Fang$^{a}$}
\\
\small $^{a}$ School of Mathematics, East China University of Science and\\
\small Technology, Shanghai 200237, China\\
\small $^{b}$ College of Science, Shihezi University, Shihezi, Xinjiang 832003, China\\}

\date{}
\maketitle {\flushleft\large\bf Abstract.}
There is a rich history of studying the existence of cycles in planar graphs. The famous Tutte theorem on the Hamilton cycle states that every 4-connected planar graph contains a Hamilton cycle. Later on, Thomassen (1983), Thomas and Yu (1994) and Sanders (1996) respectively proved that every 4-connected planar graph contains a cycle of length $n-1, n-2$ and $n-3$. Chen, Fan and Yu (2004) further conjectured that every 4-connected planar graph contains a
cycle of length $\ell$ for $\ell\in\{n,n-1,\ldots,n-25\}$ and they verified that $\ell\in \{n-4, n-5, n-6\}$.
When we remove the ``4-connected" condition, how to guarantee the existence of a long cycle in a planar graph?
A natural question asks by adding a spectral radius condition: What is the smallest constant $C$ such that for sufficiently large $n$, every graph $G$ of order $n$ with spectral radius greater than $C$ contains a long cycle in a planar graph?
In this paper, we give a stronger answer to the above question. Let $G$ be a planar graph with order $n\geq 1.8\times 10^{17}$ and $k\leq \lfloor\log_2(n-3)\rfloor-8$ be a non-negative integer, we show that if $\rho(G)\geq \rho(K_2\vee(P_{n-2k-4}\cup 2P_{k+1}))$ then $G$ contains a cycle of length $\ell$ for every $\ell\in \{n-k, n-k-1, \ldots, 3\}$ unless $G\cong K_2\vee(P_{n-2k-4}\cup 2P_{k+1})$.

\begin{flushleft}
\end{flushleft}

\section{Introduction}\label{se-1}


For a family of graphs $\mathcal{H}$, a graph is said to be $\mathcal{H}$-free if it does not contain $H\in\mathcal{H}$ as a subgraph. When $\mathcal{H}=\{H\}$ is a single graph,  we use $H$-free instead of $\mathcal{H}$-free. As usual, we denote by $K_n,C_n$ and $P_n$ the complete graph, the cycle and the path on $n$ vertices, respectively. The maximum number of edges in an $\mathcal{H}$-free graph on $n$ vertices is defined as the $\emph{Tur\'{a}n number}$ of $\mathcal{H}$, denoted by $\mathrm{ex}(n, \mathcal{H})$. Considerable focus has been directed toward the Tur\'{a}n number on cycles. F\"{u}redi and Gunderson \cite{Furedi} determined $\mathrm{ex}(n, C_{2k+1})$ for $n\geq 1$ and $2k+1\geq 5$. Ore \cite{Ore} proved that $\mathrm{ex}(n, C_n)\leq \binom{n-1}{2}+1$. However, the exact value of $\mathrm{ex}(n, C_{2k})$ is still open.

Let $G_1$ and $G_2$ be two graphs. We denote the union of $k$ disjoint copies of a graph $G$ by $k G$. The join of two disjoint graphs $G_1$ and $G_2$, denoted by $G_1 \vee G_2$, is obtained from the disjoint union $G_1\cup G_2$ by joining each vertex of $G_1$ to each vertex of $G_2$. Let $A(G)$ be the adjacency matrix of a connected graph $G$, and $\rho(G)$ be its spectral radius, i.e., the maximum modulus of eigenvalues of $A(G)$. Let $\mathrm{SPEX}(n,\mathcal{H})$ be the set of graphs on $n$ vertices with maximum spectral radius among graphs not containing a subgraph in $\mathcal{H}$. For long cycles, Fiedler and Nikiforov \cite{Fiedler} determined that $\mathrm{SPEX}(n, C_n)=\{K_1\vee(K_{n-2}\cup K_1)\}$. Only recently, Ge and Ning \cite{Ge} proved that $\mathrm{SPEX}(n, C_{n-1})=\{K_1\vee(K_{n-3}\cup K_2)\}$. For more information on relationships between the spectral radius and existence of long cycles in a graph, the readers may refer to \cite{Li-1}.


A graph is called \emph{planar} if it can be drawn in the plane with no pair of edges crossing, and such drawing is called a \emph{plane graph}. Let $\mathrm{spex}_\mathcal{P}(n, \mathcal{H})$ denote the maximum spectral radius of the adjacency matrix of any $\mathcal{H}$-free planar graphs on $n$ vertices, and $\mathrm{SPEX}_\mathcal{P}(n,\mathcal{H})$ denote the set of extremal graphs with respect to $\mathrm{spex}_\mathcal{P}(n, \mathcal{H})$. Let $t\mathcal{C}$ be the family of $t$ vertex-disjoint cycles without length restriction. Tait and Tobin \cite{Tait} proved that $K_2\vee P_{n-2}$ is the spectral extremal graph among all planar graphs with sufficiently large $n$. This implies that the extremal graphs in both $\mathrm{SPEX}_\mathcal{P}(n, tC_{\ell})$ and $\mathrm{SPEX}_\mathcal{P}(n, t\mathcal{C})$ are $K_2\vee P_{n-2}$ for $t\geq 3, \ell\geq 3$. Only very recently, Fang, Lin and Shi \cite{Fang} determined $\mathrm{spex}_\mathcal{P}(n, tC_{\ell})$ and $\mathrm{spex}_\mathcal{P}(n, t\mathcal{C})$ and characterized the unique extremal graph with sufficiently large $n$ for $1\leq t\leq 2$ and $\ell\geq 3$, respectively. Zhai and Liu \cite{Zhai-1} characterized the extremal graphs in $\mathrm{SPEX}_\mathcal{P}(n,\mathcal{H})$ when $\mathcal{H}$ is the family of $k$ edge-disjoint cycles.

Studying the existence of long cycles in planar graphs is an intriguing subject. The pioneering result in this area was established by Whitney \cite{Whitney}, demonstrating that every 4-connected planar triangulation contains a Hamilton cycle. In 1956, Tutte \cite{Tutte} extended this result to all 4-connected planar graphs. Subsequently, Thomassen \cite{Thomassen}, Thomas and Yu \cite{Thomas} and Sanders \cite{Sanders} respectively proved that every 4-connected planar graph contains a cycle of length $n-1, n-2$ and $n-3$. In 1988, Malkevitch \cite{Malkevitch} posed a conjecture concerning cycles of consecutive lengths in 4-connected planar graphs.

\begin{conj}\label{conj-1}\emph{(\cite{Malkevitch})}
Let $G$ be a 4-connected planar graph on $n$ vertices. If $G$ contains a cycle of length 4, then $G$ contains a cycle of length $\ell$ for every $\ell\in \{n, n-1, \ldots, 3\}$.
\end{conj}

Later on, Chen, Fan and Yu \cite{Chen} found a counterexample that the line graph of a cyclically 4-edge-connected cubic planar graph with girth at least 5 contains no cycle of length 4. Furthermore, they proposed the following weaker conjecture and demonstrated that every 4-connected planar graph contains a cycle of length $\ell$ for every $\ell\in \{n-4, n-5, n-6\}$.

\begin{conj}\label{conj-2}\emph{(\cite{Chen})}
Let $G$ be a 4-connected planar graph on $n$ vertices. Then $G$ contains a cycle of length $\ell$ for every $\ell\in \{n, n-1, \ldots, n-25\}$ with $\ell\geq 3$.
\end{conj}

In 2009, Cui \cite{Cui} proved the Conjecture \ref{conj-2} holds for $\ell=n-7$. Motivated by the study of the existence of cycles in graphs from the perspective of eigenvalues. Naturally, we consider the existence of a long cycle from a spectral perspective in a planar graph and pose the following problem.

\begin{prob}\label{prob-1}
What is the tight spectral radius condition for the existence of a long cycle in a planar graph?
\end{prob}

In this paper, we address Problem \ref{prob-1} by presenting preliminary findings focusing on the spectral radius. In 2008, Nikiforov \cite{Nikiforov} posed the following open problem in spectral graph theory as an analogue to the classical theorems on cycles of consecutive lengths by Bondy and Bollob\'{a}s.

\begin{prob}
What is the maximum $C$ such that for all positive $\varepsilon<C$ and sufficiently large $n$, every graph $G$ of order $n$ with $\rho(G)>\sqrt{\lfloor\frac{n^2}{4}\rfloor}$ contains a cycle of length $\ell$ for every integer $3\leq \ell \leq (C-\varepsilon)n$?
\end{prob}

The first contribution to the above problem is due to Nikiforov \cite{Nikiforov} who showed that $C\geq \frac{1}{320}$, and was improved to $C\geq \frac{1}{160}$ by Peng and Ning \cite{NP}. Only very recently, Zhai and the second author \cite{Zhai} proved that the result holds for $C\geq \frac{1}{7}$ and they further showed that ``sufficiently large $n$" condition can be deleted, Li and Ning \cite{Li}, Zhang \cite{Zhang} respectively improved these results to $C\geq \frac{1}{4}$ and $C\geq \frac{1}{3}$. Motivated by the aforementioned spectral extremal results pertaining to planar graphs, we delve into a spectral extremal problem concerning planar graphs with consecutive cycles, as stated in the following theorem.

\begin{thm}\label{thm-1}
Let $G$ be a planar graph of order $n$ and let $k\leq \lfloor\log_2(n-3)\rfloor-8$ be a non-negative integer, where $n\geq 1.8\times 10^{17}$. If $\rho(G)\geq \rho(K_2\vee(P_{n-2k-4}\cup 2P_{k+1}))$, then $G$ contains a cycle of length $\ell$ for every $\ell\in \{n-k, n-k-1, \ldots, 3\}$ unless $G\cong K_2\vee(P_{n-2k-4}\cup 2P_{k+1})$.
\end{thm}




The rest of this paper is organized as follows. In Section \ref{se-2}, we introduce some preliminaries that will be employed to prove our main result. In Section \ref{se-3}, we give the proof of Theorem \ref{thm-1}. In Section \ref{se-4}, we conclude some open problems for further study.



\section{Preliminaries}\label{se-2}

For a vertex $v\in V(G)$, the neighborhood of $v$ denoted by $N_{G}(v)=\{u: uv\in E(G)\}$ and the degree of $v$ denoted by $d_{G}(v)=|N_{G}(v)|$.
A \emph{linear forest} is a disjoint union of paths. For two non-negative integers $n$ and $a$ with $n\geq a+3$,
let $\mathcal{L}_{n,a}$ denote the family of linear forests of order $n-2$ and size $n-3-a$. For simplicity, an isolated vertex is referred to as a path of order 1. In order to obtain our main results, we first give the following lemmas.


\begin{lem}\label{lem-1}
Suppose $n$, $a_1$ and $a_2$ are three integers with $n\geq 4$ and $0\leq a_2<a_1 \leq \frac{\sqrt{2n-4}}{4}$. Let $L_i\in \mathcal{L}_{n,a_i}$ for each $i\in \{1,2\}$.
Then $\rho(K_2\vee L_2)>\rho(K_2\vee L_1)$.
\end{lem}

\begin{proof}
For each $i\in \{1,2\}$, let $E_i=E(P_{n-2})\setminus E(L_i)$.
Since $L_i\in \mathcal{L}_{n,a_i}$, we have $|E_i|=a_i$.
By the Perron-Frobenius theorem, there exists a positive eigenvector $\mathbf{x}=(x_1,x_2,\dots,x_n)^{\mathrm{T}}$ corresponding to $\rho:=\rho(K_2\vee L_1)$ with $\max_{u\in V(K_2\vee L_1)}x_u=1$. Clearly, $K_2\vee L_1$ contains exactly two dominating vertices, say $u'$ and $u''$.
Then $x_{u'}=x_{u''}=1$.

Select an arbitrary vertex $u\in V(L_1)$.
Note that $L_1\in  \mathcal{L}_{n,a_1}$. Then $d_{L_1}(u)\leq 2$, and hence

\begin{equation}\label{ineq-2.1}
2=x_{u'}+x_{u''}\leq\rho x_{u}=x_{u'}+x_{u''}+\sum\limits_{v\in N_{L_1}(u)}x_v=2+\sum\limits_{v\in N_{L_1}(u)}x_v\leq 4.
\end{equation}
Therefore, $\frac{2}{\rho}\leq x_{u}\leq \frac{4}{\rho}$, which implies that  $\sum\limits_{v\in N_{L_1}(u)}x_v\leq \frac{8}{\rho}$.
According to \eqref{ineq-2.1}, we obtain
\begin{equation}\label{ineq-2.2}
x_u\in \left[\frac{2}{\rho},\frac{2}{\rho}+\frac{8}{{\rho}^2}\right].
\end{equation}
Since $K_{2,n-2}$ is a proper subgraph of $K_2\vee L_1$, we have
     $$\rho> \rho(K_{2,n-2})=\sqrt{2n-4}\geq\max\{4a_2,2\}.$$
Combining this with \eqref{ineq-2.2}, we obtain
\begin{align*}
\sum\limits_{uv\in E_1}x_{u}x_{v}-\sum\limits_{uv\in E_2}x_ux_v
&\geq a_1\left(\frac{2}{\rho}\right)^2-a_2\left(\frac{2}{\rho}+\frac{8}{{\rho}^2}\right )^2\\
& = \frac{4(a_1-a_2)}{{\rho}^2}-\frac{8a_2}{{\rho}^3}-\frac{16a_2}{{\rho}^4}\\
&\geq \frac{4}{{\rho}^2}-\frac{2\rho}{{\rho}^3}-\frac{4\rho}{{\rho}^4}=\frac{2\rho-4}{{\rho}^3}>0.
\end{align*}
Therefore,
\begin{align*}
\rho(K_2\vee F_2)-\rho(K_2\vee F_1)
&\geq \frac{\mathbf{x}^\mathrm{T} (A(K_2\vee F_2)-A(K_2\vee F_1))\mathbf{x}}{\mathbf{x}^\mathrm{T}\mathbf{x}}\\
&\geq \frac{2}{\mathbf{x}^\mathrm{T}\mathbf{x}}\left(\sum\limits_{uv\in E_1}x_{u}x_{v}-\sum\limits_{uv\in E_2}x_ux_v\right)>0,
\end{align*}
as desired.
\end{proof}

\begin{lem}\label{lem-2}
Let $n,n_1,n_2$ and $k$ be integers with $n_1\geq n_2\geq k+2\geq 2$ and $n\geq 2^{k+8}+3$, and let $L$ be a linear forest with $|V(L)|=n-2-n_1-n_2$.
Then
  $$\rho(K_2\vee(P_{n_1+n_2-(k+1)}\cup P_{k+1}\cup L))> \rho(K_2\vee(P_{n_1}\cup P_{n_2}\cup L)).$$
\end{lem}

\begin{proof}

Assume that $P_{n_1}:=u_1u_2\cdots u_{n_1}$ and $P_{n_2}:=w_1w_2\cdots w_{n_2}$.
By the Perron-Frobenius theorem, there exists a positive eigenvector $\mathbf{x}=(x_1,x_2,\dots,x_n)^{\mathrm{T}}$ corresponding to $\rho:=\rho(K_2\vee(P_{n_1}\cup P_{n_2}\cup L))$ with $\max_{u\in V(K_2\vee(P_{n_1}\cup P_{n_2}\cup L))}x_u=1$. Since $K_{2,n-2}$ is a proper subgraph of $K_2\vee(P_{n_1}\cup P_{n_2}\cup L)$, we get $\rho> \rho(K_{2,n-2})=\sqrt{2n-4}$.
Furthermore, we have the following claim.

\begin{claim}\label{claim-3.3}
Let $i$ be a positive integer.
Set $A_i=[\frac{2}{\rho}-\frac{8\times 2^i}{\rho^2},\frac{2}{\rho}+\frac{8\times 2^i}{\rho^2}]$ and $B_i=[-\frac{8\times 2^i}{\rho^2}, \frac{8\times 2^i}{\rho^2}]$. Then\\
(i) for any $i\in\{1,\ldots,\lfloor\frac{k+2}{2}\rfloor\}$,
$\rho^i(x_{u_{i+1}}-x_{u_i})\in A_i$ and $\rho^i(x_{w_{i+1}}-x_{w_i})\in A_i$;\\
(ii) for any $i\in\{1,\ldots,\lfloor\frac{k+3}{2}\rfloor\}$, $\rho^i(x_{u_i}-x_{w_i})\in B_i$.
\end{claim}

\begin{proof}
($i$) We will proceed with the proof by using induction on $i$. Clearly,
\begin{equation}\label{ineq-15}
\rho x_{u_j}=\sum\limits_{\substack{u\sim u_j\\u\in V(G)}}x_{u}=\begin{cases}
2+x_{u_2}, &\mbox{if}~~j=1,\\
2+x_{u_{j-1}}+x_{u_{j+1}}, &\mbox{if}~~2\leq j\leq n_1-1.
\end{cases}
\end{equation}
By using a similar analysis as (\ref{ineq-2.2}), we have
\begin{equation}\label{ineq-16}
\rho(x_{u_{j+1}}-x_{u_j})=\begin{cases}
x_{u_1}+x_{u_3}-x_{u_2}\in A_1, &\mbox{if}~~j=1,\\
(x_{u_j}-x_{u_{j-1}})+(x_{u_{j+2}}-x_{u_{j-1}})\in B_1, &\mbox{if}~~2\leq j\leq n_1-2.
\end{cases}
\end{equation}
So the result is true when $i=1$. Next, assume that $2\leq i\leq \lfloor\frac{k+2}{2}\rfloor$, which implies that $k\geq 2i-2$. For $i\leq j\leq n_1-i-1$, we get $\rho(x_{u_{j+1}}-x_{u_j})=(x_{u_j}-x_{u_{j-1}})+(x_{u_{j+2}}-x_{u_{j-1}})$, and hence
\begin{equation}\label{ineq-17}
\rho^i(x_{u_{j+1}}-x_{u_j})=\rho^{i-1}(x_{u_j}-x_{u_{j-1}})+\rho^{i-1}(x_{u_{j+2}}-x_{u_{j-1}}).
\end{equation}
By the induction hypothesis, it follows that
\begin{center}
$\rho^{i-1}(x_{u_i}-x_{u_{i-1}})\in A_{i-1} \quad \text{and} \quad \rho^{i-1}(x_{u_{i+2}}-x_{u_{i-1}})\in B_{i-1}.$
\end{center}
According to (\ref{ineq-17}) and setting $j=i$, we have $\rho^{i}(x_{u_i}-x_{u_{i-1}})\in A_{i}$, as desired.
If $i+1\leq j\leq n_1-i-1$, then by the induction hypothesis,
\begin{center}
$\rho^{i-1}(x_{u_j}-x_{u_{j-1}})\in B_{i-1} \quad \text{and} \quad \rho^{i-1}(x_{u_{j+2}}-x_{u_{j-1}})\in B_{i-1}.$
\end{center}
Again by (\ref{ineq-17}), we can deduce that $\rho^{i}(x_{u_{j+1}}-x_{u_j})\in B_{i-1}$, as desired. Thus, for any $i\in\{1,\ldots,\lfloor\frac{k+2}{2}\rfloor\}$, we have
$$\rho^i(x_{u_{j+1}}-x_{u_j})\in\begin{cases}
A_i, &\mbox{if}~~j=i,\\
B_i, &\mbox{if}~~i+1\leq j\leq n_1-i-1.
\end{cases}$$
This completes the proof of $\rho^i(x_{u_{i+1}}-x_{u_i})\in A_i$.

The proof of $\rho^i(x_{w_{i+1}}-x_{w_i})\in A_i$ is similar to that of $\rho^i(x_{u_{i+1}}-x_{u_i})\in A_i$ and thus omitted here.

$(ii)$ For any $i\in\{1,\ldots,\lfloor\frac{k+3}{2}\rfloor\}$ and $j\in \{i,\ldots, n_2-i\}$, we only need to show that $\rho^i(x_{u_j}-x_{w_j})\in B_i$. Obviously
\begin{equation*}
\rho x_{w_j}=\sum\limits_{\substack{w\sim w_j\\w\in V(G)}}x_{w}=\begin{cases}
2+x_{w_2}, &\mbox{if}~~j=1,\\
2+x_{w_{j-1}}+x_{w_{j+1}}, &\mbox{if}~~2\leq j\leq n_2-1.
\end{cases}
\end{equation*}
Combining this with (\ref{ineq-2.2}) and (\ref{ineq-15}), we obtain
\begin{equation*}
\rho(x_{u_{j}}-x_{w_j})=\begin{cases}
x_{u_2}-x_{w_2}\in B_1, &\mbox{if}~~j=1,\\
(x_{u_{j-1}}-x_{w_{j-1}})+(x_{u_{j+1}}-x_{w_{j+1}})\in B_1, &\mbox{if}~~2\leq j\leq n_2-1.
\end{cases}
\end{equation*}
By induction on $i$. We have already observed that the assertion holds for $i=1$, so assume that $i\geq 2$. If $i\leq j\leq n_2-i$, then $\rho(x_{u_{j}}-x_{w_j})=(x_{u_{j-1}}-x_{w_{j-1}})+(x_{u_{j+1}}-x_{w_{j+1}})$, and hence
\begin{equation}\label{ineq-18}
\rho^i(x_{u_{j}}-x_{w_j})=\rho^{i-1}(x_{u_{j-1}}-x_{w_{j-1}})+\rho^{i-1}(x_{u_{j+1}}-x_{w_{j+1}}).
\end{equation}
By the induction hypothesis, we have
\begin{center}
$\rho^{i-1}(x_{u_{j-1}}-x_{w_{j-1}})\in B_{i-1}\quad \text{and} \quad\rho^{i-1}(x_{u_{j+1}}-x_{w_{j+1}})\in B_{i-1}.$
\end{center}
Combining this with (\ref{ineq-18}), we have $\rho^i(x_{u_{j}}-x_{w_j})\in B_i$.
\end{proof}

Since $n\geq 2^{k+8}+3$, we have $\rho\geq \sqrt{2n-4}>8\times 2^{\frac{k+3}{2}}$. For any $i\leq \frac{k+3}{2}$, we get
$$\frac{2}{\rho^{i+1}}-\frac{8\times2^i}{\rho^{i+2}}>\left(\frac{2}{\rho^{i+1}}-\frac{8\times 2^i}{\rho^{i+2}}\right)-\frac{8\times 2^i}{\rho^{i+2}}>0.$$
Combining this with Claim \ref{claim-3.3}, we obtain
\begin{align}\label{ineq-20}
x_{u_{i+1}}-x_{u_i}\geq \frac{2}{\rho^{i+1}}-\frac{8\times2^i}{\rho^{i+2}}>0
\end{align}
and
\begin{align}\label{ineq-21}
x_{u_{i+1}}-x_{w_i}=(x_{u_{i+1}}-x_{u_i})+(x_{u_i}-x_{w_i})\geq \left(\frac{2}{\rho^{i+1}}-\frac{8\times 2^i}{\rho^{i+2}}\right)-\frac{8\times 2^i}{\rho^{i+2}}>0
\end{align}
for any $i\leq \lfloor\frac{k+2}{2}\rfloor$. Similarly,
\begin{align}\label{ineq-22}
x_{w_{i+1}}>x_{w_i}~~\text{and}~~x_{w_{i+1}}>x_{u_i}~~\text{for any}~~i\leq\left\lfloor\frac{k+2}{2}\right\rfloor.
\end{align}

Denote by $H_1=P_{n_1}\cup P_{n_2}\cup L$ and $H_2=P_{n_1+n_2-(k+1)}\cup P_{k+1}\cup L$. Let $t_1$ and $t_2$ be two non-negative integers with $t_1+t_2=k+1$. Let $H^*$ be the graph obtained from $H_1$ by deleting edges $u_{t_1}u_{t_1+1}$ and $w_{t_2}w_{t_2+1}$, and adding edges $u_{t_1}w_{t_2}$ and $u_{t_1+1}w_{t_2+1}$. Note that $H^*\cong H_2$ as $t_1+t_2=k+1$. Then
\begin{align}\label{ineq-19}
\rho(K_2\vee H_2)-\rho(K_2\vee H_1)&\geq\frac{\mathbf{x}^\mathrm{T}(A(K_2\vee H_2)-A(K_2\vee H_1))\mathbf{x}}{\mathbf{x}^\mathrm{T}\mathbf{x}}\notag\\
&\geq\frac{2}{\mathbf{x}^\mathrm{T}\mathbf{x}}(x_{u_{t_1+1}}-x_{w_{t_2}})(x_{w_{t_2+1}}-x_{u_{t_1}}).
\end{align}

Next, we will divide the proof into the following two cases basing on the parity of $k$.
\begin{case}\label{case-3}
$k$ is odd.
\end{case}

Set $t_1=\frac{k+1}{2}$. Since $t_1+t_2=k+1$, it follows that $t_2=\frac{k+1}{2}$. By (\ref{ineq-21}) and (\ref{ineq-22}), we get $x_{u_{t_1+1}}> x_{w_{t_2}}$ and $x_{w_{t_2+1}}> x_{u_{t_1}}$. Combining this with (\ref{ineq-19}), we can deduce that $\rho(K_2\vee H_2)>\rho(K_2\vee H_1)$.

\begin{case}\label{case-4}
$k$ is even.
\end{case}

We first consider $x_{u_{\frac{k+2}{2}}}\geq x_{w_{\frac{k+2}{2}}}$. Let $t_1=\frac{k}{2}$. Then $t_2=\frac{k+2}{2}$ due to $t_1+t_2=k+1$.
Since $x_{u_{\frac{k+2}{2}}}\geq x_{w_{\frac{k+2}{2}}}$, it follows that $x_{u_{t_1+1}}\geq x_{w_{t_2}}$. From (\ref{ineq-22}), we can get $x_{w_{t_2+1}}> x_{w_{t_2}}$ and $x_{w_{t_2}}> x_{u_{t_1}}$. This implies that $x_{w_{t_2+1}}> x_{u_{t_1}}$. By (\ref{ineq-19}), we have $\rho(K_2\vee H_2)\geq \rho(K_2\vee H_1)$. If $\rho(K_2\vee H_2)=\rho(K_2\vee H_1)$, then $\mathbf{x}$ also is a positive eigenvector of $\rho(K_2\vee H_2)$, and hence $\rho(K_2\vee H_2)x_{w_{t_2}}=2+x_{w_{t_2-1}}+x_{u_{t_1}}$. On the other hand, $\rho(K_2\vee H_1)x_{w_{t_2}}=2+x_{w_{t_2-1}}+x_{w_{t_2+1}}$. This implies that $x_{w_{t_2+1}}=x_{u_{t_1}}$, a contradiction. Therefore, $\rho(K_2\vee H_2)> \rho(K_2\vee H_1)$.

Next, we consider $x_{u_{\frac{k+2}{2}}}< x_{w_{\frac{k+2}{2}}}$.
Let $t_1=\frac{k+2}{2}$. Then by $t_1+t_2=k+1$, we have $t_2=\frac{k}{2}$. Since $x_{u_{\frac{k+2}{2}}}< x_{w_{\frac{k+2}{2}}}$, it follows that $x_{w_{t_2+1}}> x_{u_{t_1}}$. By (\ref{ineq-20}) and (\ref{ineq-21}), we can get $x_{u_{t_1+1}}> x_{u_{t_1}}$ and $x_{u_{t_1}}> x_{w_{t_2}}$. Thus, $x_{u_{t_1+1}}> x_{w_{t_2}}$. Combining this with (\ref{ineq-19}), we have $\rho(K_2\vee H_2)> \rho(K_2\vee H_1)$.

This completes the proof.
\end{proof}

\section{Proof of Theorem \ref{thm-1}}\label{se-3}
Before proceeding, we describe some notation and terminology necessary for stating and proving results.
Let $G$ be a planar graph with vertex set $V(G)$ and edge set $E(G)$. The order and size of $G$ are denoted by $|V(G)|$ and $|E(G)|=e(G)$, respectively. For two disjoint subset $X,Y\subset V(G)$, we denote by $G[X,Y]$ the bipartite subgraph of $G$ with vertex set $X\cup Y$ and edges having one endpoint in $X$ and the other endpoint in $Y$. The subgraph of $G$ induced by $X$, denoted by $G[X]$, is the graph with vertex set $X$ and an edge set consisting of all edges of $G$ that have both ends in $X$.
Let $N_X(v):=N_G(v)\cap X$ and $d_X(v):=|N_X(v)|$.
Define $e(X,Y)$ as the number of edges in the bipartite subgraph $G[X,Y]$, and $e(X)$ as the number of edges in the subgraph $G[X]$.
Moreover,
\begin{align}\label{ineq-3.0}
e(X)\leq 3|X|-6~~\text{and}~~e(X,Y)\leq 2(|X|+|Y|)-4.
\end{align}


Denote by $\mathcal{G}_{n,k}=\underset{3\leq\ell\leq n-k}{\cup}\{~G~|~G~\text{is a}~  C_{\ell}\text{-free planar graph of order}~n\}$. Let $\mathcal{C}_{n,k}$ be the set of graphs attaining the maximum spectral radii over all graphs in $\mathcal{G}_{n,k}$. 
We first give a lemma which plays a key role in the proof of Theorem \ref{thm-1}.

\begin{lem}\label{lem-3.1}
Let $k$ be a non-negative integer and $n\geq \mathrm{max}\{1.8\times 10^{17}, 2^{k+8}+3 \}$. Then every graph in
$\mathcal{C}_{n,k}$ contains a spanning subgraph $K_2\vee (n-2)K_1$.
\end{lem}

\begin{proof}
Choose an arbitrary graph $G\in \mathcal{C}_{n,k}$ and let $\rho=\rho(G)$.
By the Perron-Frobenius theorem,
there exists a positive eigenvector $\mathbf{x}=(x_1,x_2,\dots,x_n)^{\mathrm{T}}$ corresponding to $\rho$ with $\max_{u\in V(G)}x_u=1$. Let $u'\in V(G)$ with $x_{u'}=1$.
Clearly, $K_{2,n-2}$ is planar and $C_{n-k}$-free, which implies $K_{2,n-2}\in \mathcal{G}_{n,k}$. Then
\begin{align}\label{ineq-3.1}
\rho\geq\rho(K_{2,n-2})=\sqrt{2n-4}.
\end{align}

 We proceed with a sequence of claims.

\begin{claim}\label{claim-1}
Let $M=\{u\in V(G)~|~x_{u}\geq \frac{1}{10^4}\}$. Then $|M|\leq\frac{n}{10^4}$.
\end{claim}

\begin{proof}
For each vertex $u\in V(G)$, by (\ref{ineq-3.1}) and the definition of $M$, we get
\begin{align}\label{ineq-3.2}
\frac{\sqrt{2n-4}}{10^4}\leq\rho x_u=\sum\limits_{v\in N_G(u)}x_v\leq d_{G}(u).
\end{align}
Therefore, $$|M|\frac{\sqrt{2n-4}}{10^4}\leq\sum\limits_{u\in M}d_{G}(u)\leq\sum\limits_{u\in V(G)}d_{G}(u)\leq 2(3n-6).$$
Since $n\geq 1.8\times 10^{17}$, we have $|M|\leq 3\times 10^4\sqrt{2n-4}\leq\frac{n}{10^4}$.
\end{proof}

\begin{claim}\label{claim-2}
For any $u\in M$, we have  $d_G(u)\geq (x_u-\frac{8}{10^4})n$.
\end{claim}

\begin{proof}
Since $G$ is planar, by Claim \ref{claim-1} and (\ref{ineq-3.0}), we have $e(M)\leq 3|M|\leq \frac{3n}{10^4}$, and hence
$$e(N_G(u)\setminus M,M)\leq 2(|N_G(u)\setminus M|+|M|)-4\leq 2d_G(u)+\frac{2n}{10^4}.$$
Combining the above two inequalities gives
\begin{align}\label{ineq-3.3}
\sum\limits_{v\in M}d_{N_G(u)}(v)
&=\sum\limits_{v\in M}d_{N_G(u)\cap M}(v)+\sum\limits_{v\in M}d_{N_G(u)\setminus M}(v)\nonumber\\
&\leq2e(M)+e(N_G(u)\setminus M,M)\nonumber\\
&\leq 2d_G(u)+\frac{8n}{10^4}.
\end{align}
On the other hand,
\begin{align}\label{ineq-3.4}
\sum\limits_{v\in V(G)\setminus M}d_{N_G(u)}(v)x_v\leq \sum\limits_{v\in V(G)}\frac{d_{G}(v)}{10^4}
\leq \frac{2e(G)}{10^4}\leq \frac{6n}{10^4}.
\end{align}
Combining this with (\ref{ineq-3.3}), we obtain that
\begin{align}\label{ineq-3.5}
(2n-4)x_u\leq \rho^2x_u=\sum\limits_{v\in V(G)}d_{N_G(u)}(v)x_v\leq 2d_G(u)+\frac{14n}{10^4},
\end{align}
which yields that $d_G(u)\geq (x_u-\frac{8}{10^4})n$ as $n\geq 1.8\times 10^{17}$, as desired.
\end{proof}

\begin{claim}\label{claim-3}
Assume that $u''=\max_{u\in V(G)\setminus \{u'\}}x_u$.
Then $x_{u''}\geq \frac{997}{1000}$.
\end{claim}


\begin{proof}
By Claim \ref{claim-2}, we have $d_{V(G)\setminus M}(u')\geq d_G(u')-|M|\geq (1-\frac{9}{10^4})n$.
It follows that
\begin{align}\label{ineq-3.6}
e(N_G(u')\setminus M,M\setminus\{u'\})
&= e(N_G(u')\setminus M,M)-d_{V(G)\setminus M}(u')\nonumber \\
&\leq (2n-4)-(1-\frac{9}{10^4})n\leq (1+\frac{9}{10^4})n.
\end{align}
Recall that $e(M)\leq 3|M|\leq \frac{3n}{10^4}$. Thus,
$$\sum\limits_{v\in M\setminus \{u'\}}d_{N_G(u)\cap M}(v)x_v
\leq \sum\limits_{v\in M}d_{M}(v)=2e(M)\leq \frac{6n}{10^4}. $$
Assume that $u''=\max_{u\in V(G)\setminus \{u'\}}x_u$.
Consequently,
\begin{align*}
\sum\limits_{v\in M}d_{N_G(u)}(v)x_v
&=\sum\limits_{v\in M\setminus \{u'\}}d_{N_G(u)\cap (M)}(v)x_v+\sum\limits_{v\in M\setminus \{u'\}}d_{N_G(u)\setminus M}(v)x_v+d_G(u')x_{u'}\nonumber\\
&\leq\frac{6n}{10^4}+e(N_G(u)\setminus M,M\setminus\{u'\})x_{u''}+n.
\end{align*}
Setting $u=u'$ and combining this with (\ref{ineq-3.4}), we obtain
$$2n-4\leq \rho^2x_{u'}\leq \frac{12n}{10^4}+e(N_G(u')\setminus M,M\setminus\{u'\})x_{u''}+n,$$
which leads to that $e(N_G(u')\setminus M,M\setminus\{u'\})x_{u''}\geq (1-\frac{14}{10^4})n$.
This, together with \eqref{ineq-3.6}, gives that
$x_{u''}\geq \frac{(1-\frac{14}{10^4})n}{(1+\frac{9}{10^4})n}\geq\frac{997}{10^3}$, as desired.
\end{proof}

Note that $x_{u'}=1$ and $x_{u''}\geq \frac{997}{1000}$.
By Claim \ref{claim-2}, we have
\begin{align}\label{ineq-13}
d_{G}(u')\geq \frac{999n}{1000} ~~and~~ d_{G}(u'')\geq \frac{996n}{1000}.
\end{align}
Now, let $R=N_{G}(u')\cap N_{G}(u'')$ and $S=V(G)\setminus(\{u', u''\}\cup R)$. So $|S|\leq (n-d_{G}(u'))+(n-d_{G}(u''))\leq \frac{5n}{1000}$. Next, we show the eigenvector entries of vertices in $V(G)\setminus\{u',u''\}$ are small.
\begin{claim}\label{claim-4}
Let $u\in V(G)\setminus\{u',u''\}$. Then $x_u\leq\frac{3}{100}$.
\end{claim}

\begin{proof}
We assert that for each $u\in S$, $u$ is adjacent to at most one of $u'$ and $u''$, and is adjacent to at most 2 vertices in $R$. Otherwise, $G$ would contains a copy of $K_{3,3}$, contradicting that $G$ is planar. Thus,
$$\rho\sum\limits_{u\in S}x_u\leq \sum\limits_{u\in S}d_{G}(u)\leq \sum\limits_{u\in S}(3+d_{S}(u))\leq 3|S|+2e(S)<9|S|\leq \frac{45n}{1000},$$
where the second-to-last inequality holds by $e(S)<3|S|$.
Dividing both sides by $\rho$, we get $\sum\limits_{u\in S}x_u\leq \frac{45n}{1000\rho}$.
Since $G$ is $K_{3,3}$-free, we get $u$ is adjacent to at most 4 vertices in $R\cup \{u',u''\}$ for any $u\in V(G)\setminus\{u',u''\}$.
It follows that
$$\rho x_u=\sum\limits_{w\sim u}x_w\leq 4+\sum\limits_{\substack{w\sim u\\ w\in S}}x_w\leq 4+\sum\limits_{w\in S}x_w\leq 4+\frac{45n}{1000\rho},$$
and hence $x_u\leq \frac{4}{\rho}+\frac{45n}{1000\rho^2}$. Combining this with (\ref{ineq-3.1}), we get $x_u\leq\frac{3}{100}$.
\end{proof}

\begin{claim}\label{claim-2.1}
If $G[R]\cong \cup_{i=1}^t P_{n_i}$, where $t\geq 2$ and $n_1\geq n_2\geq \cdots \geq n_t$, then $G[\{u',u''\}\cup R]$ is $C_{n-k}$-free if and only if $n_1+n_2\leq n-k-3$.
\end{claim}

\begin{proof}
We can find that the longest cycle in $K_2\vee G[R]$ is of length $n_1+n_2+2$. Moreover, $K_2\vee (P_{n_1}\cup P_{n_2})$ contains a cycle of length $\ell$ for every $\ell\in \{3, 4,\ldots, n_1+n_2+2\}$. Therefore, $n_1+n_2+2\leq n-k-1$ if and only if $K_2\vee G[R]$ is $C_{n-k}$-free, as desired.
\end{proof}

Let $\tilde{G}$ be a planar embedding of $G[\{u',u''\}\cup R]$, and let  $u_1,u_2,\ldots,u_{|R|}$ be the vertices around $u''$ in clockwise order in $\tilde{G}$ with subscripts interpreted modulo $|R|$ (see Fig. \ref{11}).

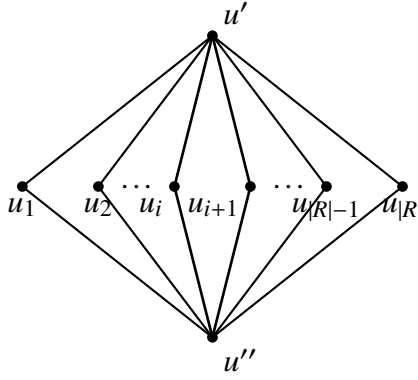
\begin{figure}[htbp!]
\centering
\begin{tikzpicture}[scale=1]
\usetikzlibrary{calc}
 \fill (0,0) circle (2pt) node[below] {$u_1$};
 \fill (1,0) circle (2pt) node[below] {$u_2$};
 \node at ($(1,0) !0.5! (2,0)$ ) {$\ldots$};

 \fill (2,0) circle (2pt) node[below left] {$u_i$};
 \fill (3,0) circle (2pt) node[below left] {$u_{i+1}$};
 \node at ($(3,0) !0.5! (4,0)$ ) {$\ldots$};
 \fill (4,0) circle (2pt) node[below] {$u_{|R|-1}$};
 \fill (5,0) circle (2pt) node[below] {$u_{|R|}$};
 \fill (2.5,2) circle (2pt) node[above right] {$u'$};
 \fill (2.5,-2) circle (2pt) node[below right] {$u''$};

 \draw[line width=0.8pt] (2.5,2) -- (0,0);
 \draw[line width=0.8pt] (2.5,2) -- (1,0);
 \draw[line width=1pt] (2.5,2) -- (2,0);
 \draw[line width=1pt] (2.5,2) -- (3,0);
 \draw[line width=0.8pt] (2.5,2) -- (4,0);
 \draw[line width=0.8pt] (2.5,2) -- (5,0);
 \draw[line width=0.8pt] (2.5,-2) -- (0,0);
 \draw[line width=0.8pt] (2.5,-2) -- (1,0);
 \draw[line width=1pt] (2.5,-2) -- (2,0);
 \draw[line width=1pt] (2.5,-2) -- (3,0);
 \draw[line width=0.8pt] (2.5,-2) -- (4,0);
 \draw[line width=0.8pt] (2.5,-2) -- (5,0);

\end{tikzpicture}
\caption{A local structure of $\tilde{G}$.}\label{11}
\end{figure}

\begin{claim}\label{claim-6}
$S$ is empty.
\end{claim}

\begin{proof}
Suppose to the contrary that $S$ is non-empty. Let $|S|=s\geq 1$. Recall that for each $u\in S$, $u$ is adjacent to at most one of $u'$ and $u''$, and is adjacent to at most 2 vertices in $R$. Since $G$ is $K_{3,3}$-minor free, we can see that $G[R]$ is $K_{1,3}$-minor free.
This indicates that $G[R]$ is either isomorphic to $C_{|R|}$, or a disjoint union of paths and isolated vertices.
Since $G[S]$ is planar,
there exists a vertex $v_1\in S$ with $d_S(v_1)\leq 5$.
Let $S_0=S$ and $S_1=S_0\setminus \{v_1\}$.
Repeat this step, we obtain a sequence of sets $S_0, S_1,\ldots,S_{s-1}$ such that $d_{S_{i-1}}(v_i)\leq 5$ and $S_{i}=S_{i-1}\setminus \{v_i\}$ for each $i\in \{1,2,\ldots,s-1\}$.
By Claims \ref{claim-3} and \ref{claim-4}, we get
\begin{align}\label{ineq-14}
\sum\limits_{\substack{w\sim v_i \\ w\in \{u',u''\}\cup R\cup S_{i-1}}}x_w
\leq 1+\sum\limits_{\substack{w\sim v_i \\ w\in R }}x_w+\sum\limits_{\substack{w\sim v_i \\ w\in S_{i-1}}}x_w
\leq \frac{121}{100}<x_{u'}+x_{u''}-\frac{7}{10}.
\end{align}

The rest of the proof will be divided into two cases according to the value of $|R|$.

\begin{case}\label{case-1}
$|R|\geq n-k-2$.
\end{case}

Since $G\in \mathcal{C}_{n,k}$ and $|R|\geq n-k-2$, it follows that $G[R]$ is a disjoint union of paths and isolated vertices. Furthermore, $G[R]$ is $P_{n-k-2}$-free.
It remains the case that $G[R]\cong \cup_{i=1}^t P_{n_i}$, where $t\geq 2$ and $n_1\geq n_2\geq \cdots \geq n_t$.
Then there exists an integer $i_0\leq |R|$ such that $u'u_{i_0}u''u_{i_0+1}u'$ is a face of $\tilde{G}$.
Let $G^*$ be the graph obtained from $\tilde{G}$ by  joining each vertex in $S$ to each vertex in $\{u',u''\}$ and making these edges cross the face $u'u_{i_0}u''u_{i_0+1}u'$.
Clearly, $G^{*}$ is planar.

Next we show that $G^{*}\in\mathcal{G}_{n,k}$.
Since $G[R]\cong \cup_{i=1}^t P_{n_i}$, we have $G^*[R\cup S]=\cup_{i=1}^t P_{n_i}\cup (|S|\cdot P_1)$. Therefore, the longest cycle in $G^{*}$ is of length $n_1+n_2+2\leq n-k-1$. By Claim \ref{claim-2.1}, we get $G^{*}$ is $C_{n-k}$-free. This indicates that
$G^*\in \mathcal{G}_{n,k}$.

One can observe that in the graph $G$ the set of edges incident to vertices in $S$ is $\cup_{i=1}^s \{wv_i|w\in N_{\{u',u''\}\cup R\cup S_{i-1}}(v_i)\}$.
Combining this with (\ref{ineq-14}), we have
\begin{align*}
\rho(G^{*})-\rho(G)&\geq\frac{\mathbf{x}^\mathrm{T}(A(G^{*})-A(G))\mathbf{x}}{\mathbf{x}^\mathrm{T}\mathbf{x}}\\
&=\frac{2}{\mathbf{x}^\mathrm{T}\mathbf{x}}\sum\limits_{i=1}^s x_{v_i}\left((x_{u'}+x_{u''})-\sum\limits_{\substack{w\sim v_i \\ w\in \{u',u''\}\cup R\cup S_{i-1}}}x_w\right)>0,
\end{align*}
contradicting that $G\in \mathcal{C}_{n,k}$.

\begin{case}\label{case-2}
$|R|\leq n-k-3$.
\end{case}
Since $G$ is planar, $G[R]$ is either isomorphic to $C_{|R|}$ or a linear forest.
Suppose first that $G[R]\cong C_{|R|}$.
Since  $G$ is planar, we have $u'u''\notin E(G)$.
Let $G^*$ be the graph obtained from $\tilde{G}$ by deleting the edges $u_1u_2, u_2u_3$, adding the edge $u'u''$, joining each vertex in $S$ to each vertex in $\{u',u''\}$ and making these edges cross the face $u'u_2u''u_3u'$.
Clearly, $G^{*}$ is planar and the longest cycle in $G^{*}$ is of length $|R|+2$.
Since $|R|\leq n-k-3$, we have $|R|+2\leq n-k-1$, which implies that $G^{*}\in \mathcal{G}_{n,k}$.
By Claim \ref{claim-4} and (\ref{ineq-14}), we get
\begin{align*}
\rho(G^{*})-\rho(G)&\geq\frac{\mathbf{x}^\mathrm{T}(A(G^{*})-A(G))\mathbf{x}}{\mathbf{x}^\mathrm{T}\mathbf{x}}\\
&\geq\frac{2}{\mathbf{x}^\mathrm{T}\mathbf{x}}\left(x_{u'}x_{u''}-x_{u_1}x_{u_2}-x_{u_2}x_{u_3}+\sum\limits_{i=1}^s x_{v_i}\left((x_{u'}+x_{u''})-\sum\limits_{\substack{w\sim v_i \\ w\in \{u',u''\}\cup R\cup S_{i-1}}}x_w\right)\right)\\
&>0,
\end{align*}
contradicting that $G\in \mathcal{C}_{n,k}$.
Thus, $G[R]$ is a linear forest. Based on this, we discuss the following in two subcases.

\begin{subcase}\label{s-case-2}
$|R|\leq n-k-4$.
\end{subcase}

If $u'u''\in E(G)$, then there exists a face $F$ in $\tilde{G}$ such that $u'u''\in E(F)$.
On the other hand, if $u'u''\notin E(G)$, then there exists an integer $i$ such that $F:=u'u_iu''u_{i+1}u'$ is a face of $\tilde{G}$. In either case, we can insert $|S|$ isolated vertices in $F$, and let $G^*$ be the graph obtained from $\tilde{G}$ by connecting each vertex in $S$ to each vertex in $\{u',u''\}$ and making these edges cross the face $F$. Clearly, $G^{*}$ is planar, and the longest cycle of $G^{*}$ is of length $|R|+3\leq n-k-1$ as $|R|\leq n-k-4$. 
Then $G^{*}\in \mathcal{G}_{n,k}$. Note that $G\subseteq P_{|R|}$ and $|R|\leq n-k-4$.
A similar discussion in Case \ref{case-1} shows that $\rho(G^{*})>\rho(G)$, a contradiction.

\begin{subcase}\label{s-case-2}
$|R|= n-k-3$.
\end{subcase}

Suppose first that $G[R]$ is a proper subgraph of $P_{n-k-3}$.
Let $G^*$ be the graph that defined as in the proof of Case 1.
Similar arguments in the proof of Case 1 show that $G^{*}\in\mathcal{G}_{n,k}$ and $\rho(G^{*})>\rho(G)$, which gives a contradiction.

It remains the case $G[R]\cong P_{n-k-3}$. 
Clearly, $s=n-2-|R|=k+1$.
We first prove that $x_{v_i}\geq \frac{1}{\rho}$ for each $v_i\in S$. Otherwise, there exists a vertex $v_{i_0}\in S$ such that $x_{v_{i_0}}< \frac{1}{\rho}$, and hence $\rho x_{v_{i_0}}=\sum\limits_{w\in N_G(v_{i_0})}x_w<1$.
Let $G^{**}$ be the graph obtained from $G$ by deleting all edges incident to $v_{i_0}$ and adding the edge $u'v_{i_0}$. Clearly, $G^{**}$ is planar and $C_{n-k}$-free, and so $G^{**}\in \mathcal{G}_{n,k}$.
However,
\begin{align*}
\rho(G^{**})-\rho(G)
&\geq\frac{\mathbf{x}^\mathrm{T}(A(G^{**})-A(G))\mathbf{x}}{\mathbf{x}^\mathrm{T}\mathbf{x}}\\
&\geq\frac{2}{\mathbf{x}^\mathrm{T}\mathbf{x}} x_{v_{i_0}}\left(x_{u'}-\sum\limits_{w\in N_G(v_{i_0})}x_w\right)>0,
\end{align*}
contradicting that $G\in \mathcal{C}_{n,k}$.
Hence, $x_{v_i}\geq \frac{1}{\rho}$ for each $v_i\in S$.

Recall that $d_R(v_i)\leq 2$ for any $v_i\in S$.
Thus, $$e(S,R)= \sum\limits_{v_i\in S}d_R(v_i)\leq 2|S|=2(k+1).$$
Let $R'$ be the set of vertices in $R$ incident to vertices in $S$.
One can observe that $|R'|\leq e(S,R)\leq 2(k+1)$,
and the subgraph $G[R\setminus R']$ contains at most $|R'|+1$ paths.
On the other hand, since $|R|=n-k-3$, we have
$$|R\setminus R'|\geq |R|-|R'|\geq (n-k-3)-2(k+1)=n-3k-5.$$
By the pigeonhole principle, we have
$$\frac{|R\setminus R'|}{|R'|+1}\geq \frac{n-3k-5}{2(k+1)}\geq 3,$$
where the last inequality holds as $n\geq 2^{k+8}+3\geq 9k+10$.
This implies that $G[R\setminus R']$ contains a path of order 3, say $P:=u_{i_0-1}u_{i_0}u_{i_0+1}$.
By the definition of $R'$, we can see that $N_S(u)=\varnothing$ for any  $u\in V(P)$.
Then,
$$\rho x_{u_i}=\sum_{u\in N_{G}(u_i)}x_u=x_{u'}+x_{u''}+\sum_{u\in N_{R}(u_i)}x_u\leq d_{G}(u_i)\leq 4$$
for each $i\in \{i_0-1,i_0,i_0+1\}$.
Consequently, $x_{u_i}\leq \frac{4}{\rho}$, and hence 
\begin{align}\label{ineq-114}
 x_{u_{i_0-1}}x_{u_{i_0}}+x_{u_{i_0}}x_{u_{i_0+1}}\leq \frac{32}{\rho^2}.
\end{align}

Let $G^{***}$ be the graph obtained from $G$ by first deleting the edges $u_{i_0-1}u_{i_0}$, $u_{i_0}u_{i_0+1}$
and all the edges incident to at least one vertex in $S$, and then adding the edges $v_iu'$ and $v_iu''$ for each $v_i\in S$ and making these edges cross the face $u'u_{i_0}u''u_{i_0+1}u'$.
Clearly, $G^{***}$ is planar and $C_{n-k}$-free, and hence $G^{***}\in \mathcal{G}_{n,k}$.
Then
\begin{align}\label{ineq-23}
&\rho(G^{***})-\rho(G)
\geq\frac{\mathbf{x}^\mathrm{T}(A(G^{***})-A(G))\mathbf{x}}{\mathbf{x}^\mathrm{T}\mathbf{x}}\nonumber\\
&\geq\frac{2}{\mathbf{x}^\mathrm{T}\mathbf{x}}
\left(\sum\limits_{i=1}^{k+1} x_{v_i}\left((x_{u'}+x_{u''})-\sum\limits_{\substack{w\sim v_i \\ w\in \{u',u''\}\cup R\cup S_{i-1}}}x_w\right)-x_{u_{i_0-1}}x_{u_{i_0}}-x_{u_{i_0}}x_{u_{i_0+1}}\right).
\end{align}
Combining this with \eqref{ineq-14}-\eqref{ineq-23}, we have
$$\rho(G^{***})-\rho(G)\geq \frac{2}{\mathbf{x}^\mathrm{T}\mathbf{x}}\left( \frac{k+1}{\rho} \frac{7}{10}-\frac{32}{\rho^2}\right)>0,$$
contradicting that $G\in \mathcal{C}_{n,k}$.

Therefore, $S$ is empty.
\end{proof}

\begin{claim}\label{claim-10}
$u'u''\in E(G)$.
\end{claim}

\begin{proof}
Suppose to the contrary that $u'u''\notin E(G)$. Note that $G\in \mathcal{C}_{n,k}$. Thus, 
$G[R]$ is $P_{n-k-2}$-free.
Then there exists some integer $i_0\in \{1,2,\dots,n-2\}$ such that $u_{i_0}u_{i_0+1}\notin E(\tilde{G}[R])$.
This implies that $u'u_{i_0}u''u_{i_0+1}u'$ is a face in $\tilde{G}$.

Let $G^{*}$ be the graph obtained from $\tilde{G}$ by adding the edge $u'u''$ and making $u'u''$ cross the face $u'u_{i_0}u''u_{i_0+1}u'$. Clearly, $G^{*}$ is a plane graph and $\rho(G^{*})>\rho(G)$. We next assert that $G^{*}\in \mathcal{G}_{n,k}$. Otherwise, $G^{*}$ contains a subgraph $H$ isomorphic to $C_{\ell}$ for every $\ell\in \{3,\dots,n-k\}$. Clearly, $u'u''\in E(H)$. Assume that $H=u'u''u_1'u_2'\ldots u_{\ell-2}'u'$.
However, an $\ell$-cycle $u'u_1'u''u_2'\ldots u_{\ell}'u'$ is already present in $G$, a contradiction. This implies that $G^{*}\in \mathcal{G}_{n,k}$. But this contradicts the maximality of $G$. Therefore, $u'u''\in E(G)$.
\end{proof}

From Claims \ref{claim-6} and \ref{claim-10},  we can see that $G$ contains a copy of $K_2 \vee (n-2)K_1$.
This completes the proof of Lemma \ref{lem-3.1}.
\end{proof}

By Lemma \ref{lem-3.1}, we find that $u'$ and $u''$ are dominating vertices of $G$, yielding $x_{u'}=x_{u''}=1$. With the above necessary tools and properties of a graph with maximum spectral radii in $\mathcal{G}_{n,k}$, we are now prepared to prove the existence of cycles of consecutive lengths from a spectral perspective.

\begin{proof}[\textbf{Proof of Theorem~\ref{thm-1}}]
Assume that $G\in \mathcal{C}_{n,k}$. By Lemma \ref{lem-3.1}, $G= K_2\vee G[R]$, where $G[R]\in \mathcal{L}_{n,a}$ for some $a\geq 0$. We first prove $a=2$. Set $G^{*}=K_2\vee (P_{n-2k-4}\cup P_{k+1}\cup P_{k+1})$.
Clearly, $P_{n-2k-4}\cup P_{k+1}\cup P_{k+1}\in \mathcal{L}_{n,2}$ and the longest cycle in $G^{*}$ is of length $n-k-1$. This indicates that $G^{*}\in \mathcal{G}_{n,k}$ and $\rho(G)\geq\rho(G^{*})$.
By Lemma \ref{lem-1}, we obtain $0\leq a\leq 2$. 
If $a\leq 1$, then $G$ contains a copy of $C_{\ell}$ for every $\ell\in\{n,n-1,\ldots,3\}$. So, $G\notin \mathcal{G}_{n,k}$, a contradiction. Hence $a=2$.

Since $G[R]\in \mathcal{L}_{n,2}$, we may assume that $G[R]\cong P_{n_1}\cup P_{n_2}\cup P_{n_3}$, where $n_1\geq n_2\geq n_3$ and $n_1+n_2+n_3=n-2$.
Since $G\in \mathcal{G}_{n,k}$, we have $n_1+n_2\leq n-k-3$, and hence $n_2\geq n_3\geq k+1$. 

Now, we prove that $n_2=k+1$.
Suppose to the contrary that $n_2\geq k+2$. Let $L'=P_{n_1'}\cup P_{n_2'}\cup P_{n_3'}$, where $n_1'=n_1+n_2-(k+1)$, $n_2'=k+1$ and $n_3'=n_3$. Clearly, $n_1'\geq n_2'\geq n_3'$ and $n_1'+n_3'=n-k-3$. By Claim \ref{claim-2.1}, $K_2\vee L'\in \mathcal{G}_{n,k}$.
However, by Lemma \ref{lem-2}, we have $\rho(K_2\vee L')>\rho(G)$,
contradicting that $G\in \mathcal{C}_{n,k}$.
Hence, $n_2=k+1$. 

Recall that $k+1\leq n_3\leq n_2=k+1$. Thus, $n_3=k+1$.
This implies that $G\cong K_2\vee (P_{n-2k-4}\cup 2P_{k+1})$, completing the proof of Theorem \ref{thm-1}.

\end{proof}


\section{Concluding remarks}\label{se-4}

The result below follows directly from Theorem \ref{thm-1} with $k=0$, and is therefore presented as a corollary without requiring a separate proof.

\begin{cor}\label{cor-2}
Let $G$ be a planar graph of order $n$ with $n\geq 1.8\times 10^{17}$. If $\rho(G)\geq \rho(K_2\vee(P_{n-4}\cup 2P_1))$. Then $G$ contains a cycle of length $\ell$ for every $\ell\in \{n, n-1, \ldots, 3\}$ unless $G\cong K_2\vee(P_{n-4}\cup 2P_1)$.
\end{cor}

Theorem \ref{thm-1} implies that for $\ell\in \left[n-\lfloor\log_2(n-3)\rfloor+8, n\right]$ and $n\geq 1.8\times 10^{17}$, we have $\mathrm{SPEX}_{\mathcal{P}}(n, C_l)=K_2\vee (P_{n-2k-4}\cup2P_{k+1})$. Nikiforov \cite{Nikiforov-1} proved that $\mathrm{SPEX}(n, C_4)=\{K_1\vee \frac{n-1}{2}K_2\}$ for $n$ is odd, while Zhai and Wang \cite{Zhai-2} proved $\mathrm{SPEX}(n, C_4)=\{K_1\vee (K_1\cup \frac{n-2}{2}K_2)\}$ for $n$ is even. Observe that $K_1\vee \frac{n-1}{2}K_2$ and $K_1\cup \frac{n-2}{2}K_2$ are both planar graphs. Thus, $\mathrm{SPEX}_{\mathcal{P}}(n, C_4)=\mathrm{SPEX}(n, C_4)$. For $\ell=3$ and $5\leq\ell\leq f(n)$, Fang, Lin and Shi \cite{Fang} characterized the spectral extremal graphs among $C_{\ell}$‐free planar graphs, where $f(n)=min\{\lfloor2(\log_2(n-3)-\log_{2}9)\rfloor+2, \lfloor \frac{8}{25}\sqrt{2(n-2)}\rfloor+2\}$. It remains mysterious to determine the spectral extremal graphs among $C_{\ell}$‐free planar graphs for some $\ell\in \left[f(n), n-\lfloor\log_2(n-3)\rfloor+8\right]$. This motivates us to propose the following problem.


\begin{prob}\label{prob-3}
For sufficiently large $n$, what are the tight spectral conditions for the existence of $C_{\ell}$ in planar graphs, where $\ell\in \left[f(n), n-\lfloor\log_2(n-3)\rfloor+8\right]$.
\end{prob}

\end{document}